\documentclass[11pt,leqno]{article}
\usepackage[margin=1in]{geometry} 
\usepackage{amssymb,amsfonts,amsmath,bbm,mathrsfs,stmaryrd,mathtools}
\usepackage{xcolor}
\usepackage{url}

\usepackage{extarrows}

\usepackage[shortlabels]{enumitem}
\usepackage{tensor}

\usepackage{xr}
\usepackage[T1]{fontenc}
\usepackage[utf8]{inputenc}
\usepackage[colorlinks,
linkcolor=black!75!red,
citecolor=blue,
pdftitle={},
pdfproducer={pdfLaTeX},
pdfpagemode=None,
bookmarksopen=true,
bookmarksnumbered=true,
backref=page]{hyperref}

\usepackage{tikz}
\usetikzlibrary{arrows,calc,decorations.pathreplacing,decorations.markings,decorations.shapes,intersections,shapes.geometric,through,fit,shapes.symbols,positioning,decorations.pathmorphing}

\usepackage{braket}

\usepackage[amsmath,thmmarks,hyperref]{ntheorem}
\usepackage{cleveref}

\newcommand\nthalias[1]{\AddToHook{env/#1/begin}{\crefalias{lemma}{#1}}}

\nthalias{definition}
\nthalias{example}
\nthalias{examples}
\nthalias{remark}
\nthalias{remarks}
\nthalias{convention}
\nthalias{notation}
\nthalias{construction}
\nthalias{sketch}
\nthalias{theoremN}
\nthalias{propositionN}
\nthalias{corollaryN}
\nthalias{lemma}
\nthalias{proposition}
\nthalias{corollary}
\nthalias{theorem}
\nthalias{conjecture}
\nthalias{question}
\nthalias{assumption}

\creflabelformat{enumi}{#2#1#3}

\crefname{section}{Section}{Sections}
\crefformat{section}{#2Section~#1#3} 
\Crefformat{section}{#2Section~#1#3} 

\crefname{subsection}{\S}{\S\S}
\AtBeginDocument{%
  \crefformat{subsection}{#2\S#1#3}%
  \Crefformat{subsection}{#2\S#1#3}% 
}

\crefname{subsubsection}{\S}{\S\S}
\AtBeginDocument{%
  \crefformat{subsubsection}{#2\S#1#3}%
  \Crefformat{subsubsection}{#2\S#1#3}% 
}

\theoremstyle{plain}

\newtheorem{lemma}{Lemma}[section]
\newtheorem{proposition}[lemma]{Proposition}

\newtheorem{theorem}[lemma]{Theorem}

\theoremstyle{plain}
\theoremnumbering{Alph}

\theoremstyle{plain}
\theorembodyfont{\upshape}
\theoremsymbol{\ensuremath{\blacklozenge}}

\newtheorem{remark}[lemma]{Remark}
\newtheorem{remarks}[lemma]{Remarks}

\newtheorem{notation}[lemma]{Notation}

\newtheorem{sketch}[lemma]{Sketch}

\crefname{definition}{definition}{definitions}
\crefformat{definition}{#2definition~#1#3} 
\Crefformat{definition}{#2Definition~#1#3} 

\crefname{ex}{example}{examples}
\crefformat{example}{#2example~#1#3} 
\Crefformat{example}{#2Example~#1#3} 

\crefname{exs}{example}{examples}
\crefformat{examples}{#2example~#1#3} 
\Crefformat{examples}{#2Example~#1#3} 

\crefname{remark}{remark}{remarks}
\crefformat{remark}{#2remark~#1#3} 
\Crefformat{remark}{#2Remark~#1#3} 

\crefname{remarks}{remark}{remarks}
\crefformat{remarks}{#2remark~#1#3} 
\Crefformat{remarks}{#2Remark~#1#3} 

\crefname{convention}{convention}{conventions}
\crefformat{convention}{#2convention~#1#3} 
\Crefformat{convention}{#2Convention~#1#3} 

\crefname{notation}{notation}{notations}
\crefformat{notation}{#2notation~#1#3} 
\Crefformat{notation}{#2Notation~#1#3} 

\crefname{table}{table}{tables}
\crefformat{table}{#2table~#1#3} 
\Crefformat{table}{#2Table~#1#3}

\crefname{lemma}{lemma}{lemmas}
\crefformat{lemma}{#2lemma~#1#3} 
\Crefformat{lemma}{#2Lemma~#1#3} 

\crefname{proposition}{proposition}{propositions}
\crefformat{proposition}{#2proposition~#1#3} 
\Crefformat{proposition}{#2Proposition~#1#3} 

\crefname{propositionN}{proposition}{propositions}
\crefformat{propositionN}{#2proposition~#1#3} 
\Crefformat{propositionN}{#2Proposition~#1#3} 

\crefname{corollary}{corollary}{corollaries}
\crefformat{corollary}{#2corollary~#1#3} 
\Crefformat{corollary}{#2Corollary~#1#3} 

\crefname{corollaryN}{corollary}{corollaries}
\crefformat{corollaryN}{#2corollary~#1#3} 
\Crefformat{corollaryN}{#2Corollary~#1#3} 

\crefname{theorem}{theorem}{theorems}
\crefformat{theorem}{#2theorem~#1#3} 
\Crefformat{theorem}{#2Theorem~#1#3} 

\crefname{theoremN}{theorem}{theorems}
\crefformat{theoremN}{#2theorem~#1#3} 
\Crefformat{theoremN}{#2Theorem~#1#3} 

\crefname{enumi}{}{}
\crefformat{enumi}{#2#1#3}
\Crefformat{enumi}{#2#1#3}

\crefname{assumption}{assumption}{Assumptions}
\crefformat{assumption}{#2assumption~#1#3} 
\Crefformat{assumption}{#2Assumption~#1#3} 

\crefname{construction}{construction}{Constructions}
\crefformat{construction}{#2construction~#1#3} 
\Crefformat{construction}{#2Construction~#1#3} 

\crefname{sketch}{sketch}{Sketches}
\crefformat{sketch}{#2sketch~#1#3} 
\Crefformat{sketch}{#2Sketch~#1#3} 

\crefname{question}{question}{Questions}
\crefformat{question}{#2question~#1#3} 
\Crefformat{question}{#2Question~#1#3} 

\crefname{equation}{}{}
\crefformat{equation}{(#2#1#3)} 
\Crefformat{equation}{(#2#1#3)}

\numberwithin{equation}{section}

\theoremstyle{nonumberplain}
\theoremsymbol{\ensuremath{\blacksquare}}

\newtheorem{proof}{Proof}
\newcommand\pf[1]{\newtheorem{#1}{Proof of \Cref{#1}}}

\newcommand\bC{{\mathbb C}}
\newcommand\bD{{\mathbb D}}

\newcommand\bR{{\mathbb R}}
\newcommand\bS{{\mathbb S}}

\newcommand\bZ{{\mathbb Z}}

\DeclareMathOperator{\id}{id}

\DeclareMathOperator{\spn}{\mathrm{spn}}

\DeclareMathOperator{\ex}{\mathrm{ex}}

\DeclareMathOperator{\cvx}{\mathrm{cvx}}

\newcommand{\cat}[1]{\textsc{#1}}

\newcommand{\qedhere}{\mbox{}\hfill\ensuremath{\blacksquare}}

\renewcommand{\square}{\mathrel{\Box}}

\newcommand{\xrightarrowdbl}[2][]{%
  \xrightarrow[#1]{#2}\mathrel{\mkern-14mu}\rightarrow
}

\title{Spans and convex combinations of boundary-valued continuous functions}
\author{Alexandru Chirvasitu}

\begin{document}

\date{}

\newcommand{\Addresses}{{% additional braces for segregating \footnotesize
  \bigskip
  \footnotesize

  \textsc{Department of Mathematics, University at Buffalo}
  \par\nopagebreak
  \textsc{Buffalo, NY 14260-2900, USA}  
  \par\nopagebreak
  \textit{E-mail address}: \texttt{achirvas@buffalo.edu}

  % % \medskip
  % % 
  % % \textsc{Department of Mathematics, INSTITUTION}
  % % \par\nopagebreak
  % % \textsc{ADDRESS}
  % % \par\nopagebreak
  % % \textit{E-mail address}: \texttt{??}
  % % 

}}

\maketitle

\begin{abstract}  
  For an $(n\ge 2)$-dimensional real Banach space $E$ with unit ball $E_{\le 1}$ and a topological space $X$ arbitrary elements in $C(X,E_{\le 1})$ are always expressible as linear combinations of at most three functions valued in the unit sphere $\partial E_{\le 1}$. On the other hand, for normal $X$, $C(X,E_{\le 1})$ can only be the convex hull of $C(X,\partial E_{\le 1})$ if the covering dimension of $X$ is strictly smaller than $\dim E$. A variant of this remark is the characterization of normal $X$ with $\dim X<\dim E$ as precisely those for which $C(X,E_{\le 1})$ is the convex hull of nowhere-vanishing continuous $X\to E_{\le 1}$ or, equivalently, that of continuous functions $X\to E_{[r,1]}$, $r\in (0,1)$ valued in arbitrarily thin spherical shells. 
  
  This extends a number of results due to Peck, Cantwell, Bogachev, Mena-Jurado, Navarro-Pascual and Jim\'enez-Vargas and others revolving around the realizability of the unit ball of $C(X,E)$ as a convex hull of its extreme points for strictly convex and/or complex $E$. 
\end{abstract}

\noindent
\emph{Key words:
  Banach space;
  Krein-Milman invariants;
  convex body;
  convex hull;
  extreme point;
  polyhedron;
  skeleton;
  topological degree
}

\vspace{.5cm}

\noindent{MSC 2020: 46E15; 52A21; 55M25; 54F45; 54C65; 52A07; 54D15; 55P05

  % 46E15 Banach spaces of continuous, differentiable or analytic functions
  % 52A21 Convexity and finite-dimensional Banach spaces (including special norms, zonoids, etc.) (aspects of convex geometry)
  % 55M25 Degree, winding number
  % 54F45 Dimension theory in general topology
  % 54C65 Selections in general topology
  % 52A07 Convex sets in topological vector spaces (aspects of convex geometry)
  % 54D15 Higher separation axioms (completely regular, normal, perfectly or collectionwise normal, etc.)
  % 55P05 Homotopy extension properties, cofibrations in algebraic topology
  
}

%\tableofcontents

%%%%%%%%%%%%%%%%%%%%%%%%%%%%%%%%
%%%%%%%%%%%%%%%%%%%%%%%%%%%%%%%%
\section*{Introduction}

Let $X$ be a topological space, $E$ a Banach space (exclusively finite-dimensional in the sequel) and $C_b(X,E)$ the uniform-normed Banach space of continuous bounded functions on the former valued in the latter. The material below revolves around representing elements in the unit ball $C_b(X,E)_{\le 1}=C(X,E_{\le 1})$ as convex and/or linear combinations of continuous functions valued in the bounding sphere $\partial E_{\le 1}$.

By way of a motivating entry point, recall \cite[Exercise 1.7]{dgls_ban_2e_1998} that for complex 1-dimensional $E$ and compact Hausdorff $X$ the function space $C(X):=C(X,\bC)$ is the span of its unit ball's \emph{extreme points} \cite[Definition V.7.1]{dgls_ban_2e_1998} (\Cref{sk:4extr} recalls one justification for the claim).

Adjacent problems have been the focus of much literature: \cite{MR1204371,MR229004,MR1458516,MR1338691,MR262804,MR1469100,MR208357,MR238281,zbMATH03225602}, for instance, examine from various angles and with varying assumptions whether and to what extent the unit ball $C(X,E_{\le 1})$ can be recovered as the (closed or plain) convex hull of its set of extreme points. Much of that literature focuses on the case of \emph{strictly convex} $E$ in the sense of \cite[Exercise 7.104(b)]{nb_tvs}: the extreme points of $C(X,E_{\le 1})$ are then precisely identifiable with the unit-sphere-valued functions, bringing the preceding discussion in scope. Strict convexity does play an important role in many of the results and techniques, manifesting principally (\cite[Proposition 3 and/or Lemma 3 ]{MR238281}, \Cref{pr:extr.pt.spn} below, etc.) as the uniqueness, for any $v\in E$ with $0<\|v\|\le 1$, of the chord of the sphere $\partial E_{\le 1}$ bisected (midway) by $v$.

Some convenient conventions, notation and terminology preamble a summary of the results below and how they relate to some of the referenced literature.

Topological spaces are always assumed Hausdorff and will always be sufficiently well-behaved in ways to be spelled out in individual results: \emph{normal}, \emph{completely regular}, \emph{paracompact} \cite[Definitions 14.8, 15.1, 20.6]{wil_top}, etc. $C(X,Y)$ denotes the space of continuous functions $X\to Y$, with $C(X):=C(X,\bC)$; this will apply mostly to compact Hausdorff $X$. For more general spaces $X$ and metric $Y$ a `$b$' subscript as in $C_b(X,Y)$ indicates boundedness (naturally, automatic for compact $X$). 

\begin{notation}\label{not:spns}
  \begin{enumerate}[(1),wide]
  \item We will consider both linear and convex combinations in (real or complex) Banach spaces. The Banach spaces of interest will be of the form $C_b(X,E)$ with $X$ (Hausdorff) normal and $E$ finite-dimensional Banach, with unit ball $E_{\le 1}$. The latter notation extends in self-explanatory ways ($E_{\le R}$, $E_{\ge 1}$, etc.), we also occasionally write
    \begin{equation*}
      \forall\left(S\subseteq \bR\right)
      \ :\ 
      E_{S}
      :=
      \left\{v\in E\ :\ \|v\|\in S\right\}.
    \end{equation*}
    The aforementioned combinations are of either \emph{boundary} points of the unit ball $C_b\left(X,E_{\le 1}\right)$ (elements $C_b\left(X,\partial E_{\le 1}\right)$) or \emph{extreme} points of that same ball (elements of $C_b(X,\ex E_{\le 1})$).

  \item Spans and \emph{convex hulls} \cite[p.2]{schn_cvx_2e_2014} are denoted by `$\spn$' and `$\cvx$' respectively, while `$\ex$' means extreme points.

  \item We can now introduce various \emph{Krein-Milman invariants}
    \begin{equation*}
      \cat{KM}^{\bullet}_{\square}
      ,\quad
      \bullet\in\{\cvx,\ \spn\}
      ,\quad
      \square\in\{\ex,\ \partial\}
    \end{equation*}
    (so named for the celebrated \cite[Theorem V.7.4]{conw_fa}, to the effect that compact convex subsets of locally convex topological vector spaces are closed convex hulls of their extreme points):

    \begin{equation*}
      \begin{aligned}
        \cat{KM}^{\bullet}_{\square;E}(X,f)
        &:=
          \inf
          \left\{
          n
          \ :\
          \exists \left(f_i\right)_{i=1}^n\subset \ex C_b\left(X,\square E_{\le 1}\right)
          ,\
          f\in\bullet\left(f_i\right)
          \right\}\\
        \cat{KM}^{\bullet}_{\square;E}(X)
        &:=
          \sup_{f\in C_b\left(X,E_{\le 1}\right)}\cat{KM}^{\bullet}_{\square;E}(X,f)\\
        \cat{KM}^{\bullet}_{\square;E}
        &:=
          \sup_{\text{compact Hausdorff }X}\cat{KM}^{\bullet}_{\square;E}(X).
      \end{aligned}    
    \end{equation*}
  \end{enumerate}
  We will be working mostly with real (finite-dimensional) Banach spaces, as it is not difficult to reduce the span-flavored results based on $n$-dimensional complex $E$ to their counterparts for $E$ regarded as a real $2n$-space instead. 
\end{notation}

Throughout, `$\dim$' refers to a space's \emph{covering dimension} \cite[Definition 1.6.7]{eng_dim}. 

\begin{theorem}\label{th:2d.3pts}
  Let $E$ be an $n$-dimensional real Banach space, $n\in \bZ_{\ge 1}$. 
  \begin{enumerate}[(1),wide]
  \item\label{item:th:2d.3pts:dim.constr} For normal Hausdorff $X$ we have
    \begin{equation*}
      \cat{KM}^{\cvx}_{\partial;E}(X)<\infty
      \xRightarrow{\quad}
      \dim X<n=\dim E.
    \end{equation*}
    
  \item\label{item:th:2d.3pts:cvx}
    We have
    \begin{equation*}
      \cat{KM}^{\cvx}_{\ex;E}
      =
      \cat{KM}^{\cvx}_{\partial;E}
      =\infty.
    \end{equation*}

  \item\label{item:th:2d.3pts:linbig.n2} We have
    \begin{equation}\label{eq:all.n.ge3}
      \cat{KM}^{\spn}_{\ex;E}
      \ge
      \cat{KM}^{\spn}_{\partial;E}
      \ge
      3.
    \end{equation}
    
  \item\label{item:th:2d.3pts:lin.n2} For $n\ge 2$ the second inequality in \Cref{eq:all.n.ge3} is in fact an equality. In particular, if $E$ is furthermore strictly convex then \Cref{eq:all.n.ge3} is a chain of equalities. 
  \end{enumerate}
\end{theorem}

\Cref{th:2d.3pts}\Cref{item:th:2d.3pts:dim.constr} generalizes \cite[Theorem 2]{MR238281} and (assuming normality) the relevant implications of \cite[Corollary 10]{MR1469100}, \cite[Corollary 9]{MR1338691} or \cite[Theorem II]{MR229004} by dropping the strict convexity assumption on $E$ (or alternatively, dropping the complex structure on $E$, needed in \cite[Theorem 12, (3) $\Rightarrow$ (7)]{MR1458516}). Similarly, \Cref{th:2d.3pts}\Cref{item:th:2d.3pts:lin.n2} dispenses with the complex structure on $E$ employed in \cite[Corollary 9]{MR1458516}. 

We remind the reader (e.g. \cite[Corollary 8]{MR1204371}, \cite[proof of Theorem 3]{MR238281}, the very similar \cite[proof of Theorem II]{MR229004}, etc.) that for
\begin{itemize}[wide]
\item normal $X$ with $\dim X<n$;

\item and $n$-dimensional $E$;

\item arbitrary continuous maps $X\to E_{\le 1}$ are averages of two \emph{non-vanishing} such maps. 
\end{itemize}
The proof of \Cref{th:2d.3pts}\Cref{item:th:2d.3pts:dim.constr} can be leveraged into a criterion \emph{equivalent} to $\dim<n$ (for normal spaces), again with no strict-convexity assumptions.

\begin{theorem}\label{th:dimn.iff}
  For a normal space $X$ the following conditions are equivalent.
  \begin{enumerate}[(a),wide]
  \item\label{item:th:dimn.iff:dim} $\dim X<n$.

  \item\label{item:th:dimn.iff:all.shl} For every $n$-dimensional Banach space $E$ and any (some) $0<r<1$ every continuous function $X\to E_{\le 1}$ is a convex combination of continuous functions $X\to E_{[r,1]}$.

  \item\label{item:th:dimn.iff:1.shl} As in \Cref{item:th:dimn.iff:all.shl}, for any one fixed $E$.

  \item\label{item:th:dimn.iff:all.e.2} For every $n$-dimensional Banach space $E$ every continuous function $X\to E_{\le 1}$ is the average of two nowhere-0 such functions.
    
  \item\label{item:th:dimn.iff:all.e.n} For every $n$-dimensional Banach space $E$ every continuous function $X\to E_{\le 1}$ is a convex hull of nowhere-0 such functions.

  \item\label{item:th:dimn.iff:1.e.2} As in \Cref{item:th:dimn.iff:all.e.2}, for any one fixed $E$. 

  \item\label{item:th:dimn.iff:1.e.n} As in \Cref{item:th:dimn.iff:all.e.n}, for any one fixed $E$.
  \end{enumerate}
\end{theorem}

%%%%%%%%%%%%%%%%%%%%%%%%%%%%%%%%
\subsection*{Acknowledgments}

I am grateful for comments, pointers and suggestions from B. Badzioch, I. Thompson and J. Xia.

% % %%%%%%%%%%%%%%%%%%%%%%%%%%%%%%%%
% % %%%%%%%%%%%%%%%%%%%%%%%%%%%%%%%%
% % \section{Preliminaries}\label{se:prel}

%%%%%%%%%%%%%%%%%%%%%%%%%%%%%%%%
%%%%%%%%%%%%%%%%%%%%%%%%%%%%%%%%
\section{Spans of extreme and/or boundary points in function spaces}\label{se:extr.spn}

We use standard notation (e.g. \cite[Conventions and notation]{schn_cvx_2e_2014}) for segments in convex sets:
\begin{equation*}
  [p,q]
  :=
  \left\{\lambda p+(1-\lambda)q\ :\ \lambda\in [0,1]\right\},
\end{equation*}
extending in the expected manner to $[p,q)$, $(p,q)$, etc.

\cite[Exercise 1.7]{dgls_ban_2e_1998} states that the $C^*$-algebra $C(X)$ spanned by its unit ball's extreme points. One approach that will confirm this, exhibiting an arbitrary element of $C(X)$ as a linear combination of at most 4 extreme points, would run as follows.

\begin{sketch}\label{sk:4extr}
  \begin{enumerate}[(1),wide]
  \item Observe first that every $f\in C(X)$ is the sum of two nowhere-zero functions
    \begin{equation*}
      f
      =
      k+f+(-k)
      ,\quad
      \|f\|<k\in \bR_{>0}.
    \end{equation*}
    This reduces the problem to showing that every nowhere-zero $f$ is a linear combination of at most 2 extreme points.

  \item Scaling $f$ if necessary, we can always assume it to take values in
    \begin{equation*}
      \bD_{\ge r}
      :=
      \left\{z\in\text{unit disk }\bD\ :\ |z|\ge r\right\}.
    \end{equation*}

  \item As such a function decomposes as $\id_{\bD_{\ge r}}\circ f$, it suffices to substitute for $(X,f)$ the concrete space-function pair $\left(\bD_{\ge r},\id\right)$.

  \item For every point $z\in \bD_{\ge r}$ write
    \begin{itemize}[wide]
    \item $f_i(z)$, $i=1,2$ for the two endpoints of the arc on $\bS^1=\partial \bD$ having $z$ as its midpoint;

    \item with the triple $(0,f_1(z),f_2(z))$ oriented counter-clockwise.          
    \end{itemize}
    \item $\id_{\bD_{\ge r}}$ being the average of $f_i$, $i=1,2$, we are done. 
  \end{enumerate}
\end{sketch}

% % It will be convenient to have some shorthand notation for the various invariants featuring in \Cref{sk:4extr} and its ilk.
% %

\begin{remarks}\label{res:not.if.few.ex}
  \begin{enumerate}[(1),wide]
  \item\label{item:res:not.if.few.ex:few.ex} Observe that there is no reason why \cite[Exercise 1.7]{dgls_ban_2e_1998} would hold, generally, for $C(X,E)$ with arbitrary finite-dimensional Banach spaces $E$: the unit ball of $E$ might well have only finitely many extreme points (e.g. the finite-dimensional $\ell^1$ or $\ell^{\infty}$ spaces), in which case
    \begin{equation*}
      \ex C(E_{\le 1},\ E)_{\le 1}
      =
      C(E_{\le 1},\ \ex E_{\le 1})
    \end{equation*}
    will consist of constant functions on the unit ball $E_{\le 1}$.

  \item\label{item:res:not.if.few.ex:disk.univ} One general principle is apparent in examining \Cref{sk:4extr}: the unit disk $\bD$ is in a sense universal for providing bounds on the number of extreme points required to reconstruct any other function as a linear combination thereof. \Cref{le:disk.univ} notes the immediate generalization; we frequently appeal to it tacitly, constraining consideration of $\cat{KM}$ invariants to identity functions $\id_{E_{\le 1}}$.

    Observe also that the argument proving it also shows that alternative definitions of $\cat{KM}^{\bullet}_{\square;E}$ involving broader classes of spaces (e.g. normal) would have yielded the same invariant. 
  \end{enumerate}  
\end{remarks}

\begin{lemma}\label{le:disk.univ}
  For a finite-dimensional real or complex Banach space $E$ we have
  \begin{equation*}
    \cat{KM}^{\bullet}_{\square;E}
    =
    \cat{KM}^{\bullet}_{\square;E}\left(E_{\le 1},\ \id\right).
  \end{equation*}
  \qedhere
\end{lemma}

In light of \Cref{res:not.if.few.ex}\Cref{item:res:not.if.few.ex:few.ex}, we record in \Cref{pr:extr.pt.spn} what \Cref{sk:4extr} effectively proves. Recall that
\begin{itemize}[wide]
\item a \emph{convex body} \cite[\S 16.4]{k_tvs-1} (in a topological vector space) is a convex neighborhood of one of its points (always assumed compact in the finite-dimensional case of interest here, following the terminology of \cite[post Theorem 1.1.15]{schn_cvx_2e_2014});

\item a convex body is \emph{strictly convex} \cite[\S 26.1]{k_tvs-1} if its extreme points are precisely its boundary points (equivalently: the boundary contains no non-degenerate segments);

\item and a normed space is strictly convex when its unit ball is.
\end{itemize}
We take for granted the correspondence
\begin{equation*}
  K
  \xleftrightarrow{\quad}
  \|\cdot\|_K
  \quad
  \left(\text{\emph{Minkowski functional} \cite[post Lemma 1.7.12]{schn_cvx_2e_2014}}\right)
\end{equation*}
between 0-symmetric convex bodies $K\subset \bR^n$ and Banach-space structures on $\bR^n$: $\|\cdot\|_K$ is a norm recovering $K$ as its unit ball. 

\begin{proposition}\label{pr:extr.pt.spn}
  For compact Hausdorff $X$ and 2-dimensional strictly-convex real Banach spaces $E$ every element of $C(X,E)$ is a linear combination of at most four extreme points of its unit ball. 
\end{proposition}
\begin{proof}
  Observe that \Cref{sk:4extr} goes through verbatim, with strict convexity playing a dual role.
  \begin{itemize}[wide]

  \item First, by its very definition we have $\ex E_{\le 1}=\partial E_{\le 1}$. 
    
  \item Secondly, strict convexity also ensures that points
    \begin{equation*}
      p\in
      E_{[r,1]}
      :=
      \left\{v\in E\ :\ r\le \|v\|\le 1\right\}
      ,\quad
      r>0
    \end{equation*}
    are midpoints of \emph{unique} chords $c(p)$ of $\partial E_{\le 1}=\ex E_{\le 1}$. \cite[Lemma 3]{MR238281} confirms this, but also: if $P$ is a parallelogram centered at $p$ and hence \emph{not} at the origin then the convex hull $\mathrm{cvx}\left(P\cup (-P)\right)$ is a 2-planar affine projection of a 3-cube, so has  at most 6 edges (as follows easily from either \cite[Theorem 7.16 or Lemma 7.10]{zieg_polyt_1995}). In particular, the vertices of $\pm P$ cannot all be extreme points. This uniqueness then renders $c(\bullet)$ continuous, finishing the proof per \Cref{sk:4extr}.
  \end{itemize}
\end{proof}

It is natural, at this point, to ask whether the bound of four obtained in \Cref{pr:extr.pt.spn} is optimal; \Cref{th:2d.3pts} shows that it is not. Observe also the contrast between \emph{linear} and \emph{convex} combinations: the bound is strictly higher (indeed, infinite) in the latter case. 

%\newpage

\pf{th:2d.3pts}
\begin{th:2d.3pts}
  \begin{enumerate}[label={},wide]

  \item \textbf{\Cref{item:th:2d.3pts:dim.constr}} \cite[Theorem II, proof's last paragraph]{MR229004} handles the analogous statement for Hilbert-space valued maps; an examination thereof makes it clear that the crucial property there is strict convexity, so we alter that argument appropriately to obviate the need for that constraint.

    One of the main ingredients is an appeal to the characterization (\cite[Theorem 3.5 and its proof]{MR20771}, \cite[Theorem 3.1]{zbMATH03098536}) of $(<n)$-dimensional spaces via extension properties for maps into spheres: for normal $X$, $\dim X<n$ precisely continuous maps $A\xrightarrow{f} \bS^{n-1}$ defined on arbitrary closed $A\subseteq X$ extend continuously across all of $X$.

    Given such an $f$, regarded as taking values in $\partial E_{\le 1}$:
    \begin{itemize}[wide]
    \item extend to a continuous map $X\to E_{\le 1}$ by \emph{Tietze} \cite[Theorem 3.2]{zbMATH03098536};

    \item write
      \begin{equation*}
        f
        =
        \sum_{i\ge 0}^{\substack{\text{convex}\\\text{combination}\\\lambda_i>0}} \lambda_i f_i
        ,\quad
        X
        \xrightarrow{\quad f_i\quad}\partial
        E_{\le 1};
      \end{equation*}

    \item specializing back to $A$, where $f$ takes values in $E_{\le 1}$, observe that for every $a\in A$ all $f_i(a)$, $f(a)$ must belong to a \emph{support hyperplane} \cite[\S 1.3]{schn_cvx_2e_2014} of $E_{\le 1}$ at $f(a)$ so in particular
      \begin{equation*}
        [f(a),f_0(a)]
        \subseteq
        H\cap E_{\le 1}=H\cap \partial E_{\le 1};
      \end{equation*}

    \item as $f|_A$ can be homotoped onto $f_0|_A$ with uniform speed along the segments $[f(a),f_0(a)]$, the two are \emph{uniformly homotopic} in the sense of \cite[\S 2, p.204]{MR20771};

    \item per \emph{Borsuk's} \cite[Theorem 2.1]{MR20771}, $f|_A$ extends to $X\to \partial E_{\le 1}$ if and only if $f_0|_A$ does;

    \item the later condition of course obtains, since $f_0$ took values in $\partial E_{\le 1}$ to begin with.       
    \end{itemize}
    
  \item \textbf{\Cref{item:th:2d.3pts:cvx}} This of course follows from \Cref{item:th:2d.3pts:dim.constr} and \Cref{le:disk.univ}, given that
    \begin{equation*}
      \ex E_{\le 1}\subseteq \partial E_{\le 1}
      \xRightarrow{\quad}
      \cat{KM}^{\cvx}_{\ex;E}
      \ge
      \cat{KM}^{\cvx}_{\partial;E}
    \end{equation*}
    and $\dim E_{\le 1}=n$ (as opposed to being strictly smaller); we provide an alternative, direct proof. Since %l
    \begin{equation}\label{eq:incl.ineq}
      \ex E_{\le 1}\subseteq \partial E_{\le 1}
      \xRightarrow{\quad}
      \cat{KM}^{\cvx}_{\ex;E}
      \ge
      \cat{KM}^{\cvx}_{\partial;E}
      \xlongequal[\quad]{\quad\text{\Cref{le:disk.univ}}\quad}
      \cat{KM}^{\cvx}_{\partial;E}\left(E_{\le 1},\ \id\right),
    \end{equation}
    it will suffice to prove the latter infinite. 
    
    If
    \begin{equation*}
      \forall\left(v\in E_{\le 1}\right)
      \ :\
      v=\sum_{i}\alpha_i f_i(v)
      ,\quad
      \text{fixed}
      \left[
        \begin{aligned}
          &\alpha_i\ge 0,\ \sum_i \alpha_i=1\\
          &E_{\le 1}
            \xrightarrow[\quad\text{continuous}\quad]{\quad f_i\quad}
            \partial E_{\le 1}
        \end{aligned}
      \right.
    \end{equation*}
    then each individual $f_i$ will fix the extreme points of $E_{\le 1}$, thus restricting by \Cref{pr:deg.fix.ex} to a degree-1 self-map of $\bS^{n-1}\cong \partial E_{\le 1}$. The assumed extensibility of $f|_{\partial E_{\le 1}}$ across all of $E_{\le 1}$, on the other hand, renders $f|_{\partial E_{\le 1}}$ homotopic to a constant self-map of the sphere. Thus the contradiction:
    \begin{equation*}
      1=
      \deg\left(f|_{\partial E_{\le 1}}\right)
      \xlongequal[\quad]{\ \text{\cite[Proposition II.8.2]{hu_homol_1966}}\ }
      \deg\left(\text{constant}\right)
      =0.
    \end{equation*}

    % \newpage
    
  \item \textbf{\Cref{item:th:2d.3pts:linbig.n2}} The first inequality in \Cref{eq:all.n.ge3} is automatic (as observed in \Cref{eq:incl.ineq} for convex rather than linear combinations), so we confirm the second (again in the context of the identity $\id_{E_{\le 1}}$, via \Cref{le:disk.univ}). 

    Suppose
    \begin{equation*}
      \forall\left(v\in E_{\le 1}\right)
      \ :\ 
      v=\alpha_1 f_1(v) + \alpha_2 f_2(v)
      ,\quad
      E_{\le 1}
      \xrightarrow[\quad\text{continuous}\quad]{\quad f_{1,2}\quad}
      \partial E_{\le 1}
    \end{equation*}
    for fixed $\alpha_i\in \bR$. Scaling and changing signs, this can be recast as 
    \begin{equation*}
      \forall\left(v\in E_{\le 1}\right)
      \ :\ 
      v=\lambda \varphi_1(v) + (1-\lambda) \varphi_2(v)
      ,\quad
      E_{\le 1}
      \xrightarrow[\quad\text{continuous}\quad]{\quad \varphi_{1,2}\quad}
      \partial E_{\le R}      
    \end{equation*}
    for some $R>0$ and $\lambda\in (0,1)$. Specializing at $v:=0$ forces $\lambda=\frac 12$, and it remains to argue (scaling the entire discussion back inside $E_{\le}$) that the endpoints of a chord on $\partial E_{\le 1}$ bisected midway by $v$ cannot be chosen continuously for $v$ ranging over a neighborhood $E_{<\varepsilon}\ni 0$. We relegate this to \Cref{th:no.cont.sel.chrd}, which proves more: there is no continuous selection
    \begin{equation*}
      E_{<\varepsilon}\ni
      v
      \xmapsto{\quad}
      \left(\text{$v$-bisected chord of }\partial E_{\le 1}\right)
    \end{equation*}
    valued in the \emph{Vietoris-topologized} \cite[\S 1.2]{ct_vietoris} space of compact subsets of $E$. 
    
  \item \textbf{\Cref{item:th:2d.3pts:lin.n2}} The second claim does indeed follow from the first, given that strict convexity means \emph{precisely} $\ex E_{\le 1} = \partial E_{\le 1}$. The goal for the duration will thus be to show that
    \begin{equation*}
      \forall\left(E,\ n:=\dim E\ge 2\right)
      \ :\
      \cat{KM}^{\spn}_{\partial;E}
      =
      3.
    \end{equation*}
    Given \Cref{eq:all.n.ge3}, this in turn amounts to showing that the identity on $E_{\le 1}$ is a linear combination of at most three $\partial E_{\le 1}$-valued functions.

    The various length estimates are with respect to the standard Hilbert norm $\|\cdot\|$ on $\bR^n$, wherein we regard $E_{\le 1}$ as a 0-symmetric convex body in $\bR^n$. 

    \begin{enumerate}[(I),wide]
    \item\textbf{Conclusion, conditionally.} We will first sketch the construction assuming a number of choices have been made judiciously, and then unwind the contextual meaning of `judicious'. 

      \begin{itemize}[wide]
      \item Consider first a non-empty open subset $U\subset E_{<1}$, positioned so as to ensure the existence of continuous functions
        \begin{equation}\label{eq:f23}
          U
          \xrightarrow{\quad f_{2,3}\quad}
          \partial E_{\le 1}
          ,\quad
          \forall\left(v\in U\right)
          \ :\
          v=\frac{f_2(v)+f_3(v)}{2}.
        \end{equation}

      \item For some conveniently large $R>0$, a $\lambda\in (0,1)$ and some $p\in R\partial E_{\le 1}$ the $p$-centered \emph{homothety} \cite[\S 5.1]{cox_intro-geom_2e_1969}
        \begin{equation*}
          \bR^n\ni
          v
          \xmapsto{\quad T\quad}
          \frac {1}{\lambda}v
          -
          \frac{1-\lambda}{\lambda} p
          \in \bR^n
        \end{equation*}
        maps $E_{\le 1}$ inside $R U$.

      \item Set
        \begin{equation*}
          E_{\le 1}\ni
          v
          \xmapsto[\quad\text{constant}\quad]{\quad f_1\quad}
          \frac{1}{R} p
          \in
          \partial E_{\le 1}.
        \end{equation*}

      \item By construction:
        \begin{equation*}
          \forall\left(v\in E_{\le 1}\right)
          \ :\
          v=\lambda Tv + (1-\lambda) p
          =
          R\left(
            (1-\lambda)f_1(v)
            +
            \frac{\lambda}{2}f_2\left(\frac{Tv}{R}\right)
            +
            \frac{\lambda}{2}f_3\left(\frac{Tv}{R}\right)
          \right).
        \end{equation*}
      \end{itemize}
      To close the remaining gap, we need to argue that an open $U\ne \emptyset$ admitting the $f_{2,3}$ of \Cref{eq:f23} does in fact exist.

    \item\textbf{A convenient $U$.} Recall (paraphrasing \cite[post Lemma 1.4.6]{schn_cvx_2e_2014}) that an \emph{exposed} point $q\in \partial E_{\le 1}$ is one where some linear functional $\bR^n\xrightarrow{\varphi}\bR$ achieves its \emph{unique} minimum on $E_{\le 1}$:
      \begin{equation*}
        \forall\left(E_{\le 1}\ni q'\ne q\right)
        \ :\ 
        \varphi(q')>\varphi(q)=m:=\min_{E_{\le 1}}\varphi.
      \end{equation*}
      These abound ($E_{\le 1}$ being the closed convex hull of its exposed points \cite[Corollary 1.4.5 and Theorem 1.4.7]{schn_cvx_2e_2014}), so we can certainly fix one such along with the auxiliary $\varphi$ producing it.

      For $v\in \bR^n$ off the line spanned by $q$ write
      \begin{equation*}
        K_{v,q}
        :=
        K\cap \spn\left\{v,q\right\}
        ,\quad
        K:=E_{\le 1}
      \end{equation*}
      (a 2-planar section of $K$). Observe that for sufficiently small $\varepsilon>0$ (henceforth fixed) the portion $\partial K_{v,q}\cap \varphi^{-1}([m,m+\varepsilon])$ of the relative boundary of $K_{v,q}$ will not contain the vertices of a parallelogram for any $v$. Now ensure that
      \begin{itemize}[wide]
      \item $U$ avoids the segment $[q,-q]$;

      \item and is contained in a neighborhood of $q$ sufficiently small that the endpoints of the (unique, by the choice of $\varepsilon$) chord of $K_{v,q}$ bisected by $v\in U$ is contained in $\partial K_{v,q}\cap \varphi^{-1}([m,m+\varepsilon])$. 
      \end{itemize}
      This is sufficient for our purposes: $f_{2,3}$ will be the endpoints of the bisected chord just mentioned, with (say) $f_2(v)$ and $v\in U$ on opposite sides of the line $\overline{q,-q}$ in the plane $\spn\{v,q\}$.
    \end{enumerate}
    This concludes the proof of the theorem.  \qedhere
  \end{enumerate}
\end{th:2d.3pts}

Recall (\cite[\S 6.5]{td_alg-top}, \cite[\S 2.2]{hatch_at}, \cite[\S II.8]{hu_homol_1966}) that the \emph{degree} of a self-map of $\bS^{n-1}$ is the scaling factor it induces on $H_{n-1}\left(\bS^{n-1},\bZ\right)\cong \bZ$.

\begin{proposition}\label{pr:deg.fix.ex}
  Let $K\subset \bR^n$ be a convex body and
  \begin{equation*}
    \bS^{n-1}\cong \partial K
    \xrightarrow[\quad\text{continuous}\quad]{f}
    \partial K.
  \end{equation*}
  If $f$ fixes the extreme points of $K$ then $\deg f=1$. 
\end{proposition}
\begin{proof}
  The argument is two-pronged.

  \begin{enumerate}[(I),wide]
  \item\textbf{: Reduction to polyhedra.} Substituting for $K$ convex hulls
    \begin{equation*}
      P_F:=\mathrm{cvx}\; F
      ,\quad
      \text{finite }F\subseteq \ex K
    \end{equation*}
    will produce, for sufficiently large $F$, polyhedra arbitrarily close to $K$ in the \emph{Hausdorff metric} \cite[Definition 7.3.1]{bbi} (equivalently \cite[Note 1.8.2]{schn_cvx_2e_2014}, in the Vietoris topology on the space of closed subsets of $\bR^n$). This is immediate, for instance, upon an examination of \cite[Theorem 1.8.16]{schn_cvx_2e_2014} (and its proof).
    
    Projection along rays based at the origin will homotope such closely-approximating $P_F$ onto $K$, and transport self-maps of their boundaries onto one another. The degree being a homotopy invariant \cite[Proposition II.8.2]{hu_homol_1966}, this indeed effects the desired reduction of the claim from $K$ to $P_F$. 
    
  \item\textbf{: Polyhedral $K$.} The restrictions of $f$ to the \emph{$r$-skeletons} $\ex_r K$ \cite[post Remark 2.1.5]{schn_cvx_2e_2014} can be homotoped to the identity recursively on $0\le r\le n-1$, relative smaller skeletons:
    \begin{itemize}[wide]
    \item in first instance, independently over each edge module that edge's endpoints;
      
    \item independently over each 2-face relative to that 2-face's boundary afterwards;

    \item and in general, independently over each $d$-face relative to that face's $(d-1)$-dimensional boundary. 
    \end{itemize}
  \end{enumerate}
\end{proof}

In \Cref{th:no.cont.sel.chrd} we write
\begin{equation*}
  \forall\left(\text{set }X\right)
  \ :\ 
  X^{[m]}
  :=
  X^m/S_m
  ,\quad
  S_m:=\text{symmetric group acting in the obvious fashion}
\end{equation*}
(the $m^{th}$ \emph{symmetric power/product} of $X$ \cite[Definition 78.2]{gorn_top-fx-pt_2e_2006}) equipped with the \emph{quotient topology} \cite[\S 22]{mnk}. 

\begin{remark}\label{re:sym.pow}
  There is a further continuous surjection
  \begin{equation*}
    X^{[m]}
    \xrightarrowdbl{\quad}
    X_{[m]}
    :=
    \left\{Z\subseteq X\ :\ 1\le |Z|\le m\right\}
  \end{equation*}
  onto the space of non-empty $(\le m)$-element subsets, equipped with the Vietoris topology. $X_{[m]}$ is also, confusingly, sometimes \cite{MR1562283,MR3338733,MR3856595} referred to as a symmetric product and/or power. There is no distinction between the two at $m=2$, which is the case of interest below. 
\end{remark}

\begin{theorem}[cf. {\cite[Proposition 9]{MR1469100}}]\label{th:no.cont.sel.chrd}
  For a 0-centered convex body $K\subseteq \bR^n$ any section $U\xrightarrow{s} \left(\partial K\right)^{[2]}$ of the midpoint map
  \begin{equation}\label{eq:midpt}
    \left(\partial K\right)^{[2]}
    \ni
    (p_1,p_2)
    \xmapsto{\quad\mathrm{mid}\quad}
    \frac{p_1+p_2}2
    \in K
  \end{equation}
  defined on a neighborhood $U\ni 0\in K$ must be discontinuous at $0$. 
\end{theorem}
\begin{proof}
  Consider a chord $[p_0,p_1]$ of $\partial K$ bisected by a point $p$ in the punctured neighborhood $U^{\times}:=U\setminus \{0\}$. The origin then bisects $[p_0,q:=-p_0]$. We have
  \begin{equation}\label{eq:dbl.chrd}
    [0,p]\ \|\ [q,p_1]
    \quad\text{and}\quad
    \|q-p_1\| = 2\|p\|.
  \end{equation}
  If $U$ is small enough to ensure that $\|p\|\ll \|p'\|$ for
  \begin{equation*}
    [p',-p']
    :=
    \text{$\partial K$-chord containing }[0,p]
    =
    \left(\text{line }\overline{0,p}\right)\cap K
  \end{equation*}
  (which of course we can always assume), the length equation in \Cref{eq:dbl.chrd} ensures a uniform (in $p$, $p_0$ and $p_1$) positive lower bound on the distance between the parallel lines $\overline{0,p}$ and $\overline{q,p_1}$. There is of course also a global \emph{upper} bound on the (Hausdorff, say) distance between the segments $[0,p]$ and $[q,p_1]$. Denoting by $m\angle(\ell,\ell')$ the $\left[0,\frac{\pi}2\right]$-valued angle between two lines $\ell,\ell'$, a simple elementary-geometry exercise in examining the trapezoid $(0,p,q,p_1)$ yields a global lower bound
  \begin{equation}\label{eq:tht.bd}
    \exists\left(\theta>0\right)
    \forall\left(U^{\times}\ni p \text{ bisecting } [p_0,p_1]\right)
    \ :\ 
    m\angle\left(\overline{0,p},\ \overline{p_0,p_1}\right)>\theta.
  \end{equation}
  This suffices: if the locally-defined section (i.e. right inverse to $\mathrm{mid}|_{\mathrm{mid}^{-1}(U)}$) $U\xrightarrow{s} \left(\partial K\right)^{[2]}$ were continuous at 0, then on the one hand for $p$ sufficiently close to 0 the angle between $s(p)$ and $[0,p]$ would be arbitrarily small, while also bounded below by $\theta>0$ by \Cref{eq:tht.bd}. 
\end{proof}

\begin{remark}\label{re:cont.bis}
  In reference to continuous sections of \Cref{eq:midpt}, i.e. choosing sphere chords bisected by $p$ continuously in $p$, the referenced literature feature a number of positive results that instantiate a broader pattern. A sampling, assuming $E$ strictly-convex Banach throughout:
  \begin{enumerate}[(a),wide]
  \item\label{item:re:cont.bis:ball.min.segm} If $\dim E\ge 2$ the midpoint map
    \begin{equation}\label{eq:big.mdpt.map}
      \left(\partial E_{\le 1}\right)^2
      \ni
      (p_1,p_2)
      \xmapsto{\quad}
      \frac{p_1+p_2}2
      \in
      E_{\le 1}
    \end{equation}
    splits \cite[Proposition 3]{MR1338691} over any set of the form $E_{\le 1}\setminus [0,q]$ with $\|q\|=1$. 

  \item\label{item:re:cont.bis:ball.min.segm.bis} Effectively the same principle drives the proof of \cite[Lemma 2]{MR229004}, which uses a splitting of \Cref{eq:big.mdpt.map} over the same type of set (ball minus radial segment), albeit assuming $E$ finite-dimensional Hilbert. 
    
  \item\label{item:re:cont.bis:shll} In like spirit, \cite[Proposition 9]{MR1204371} (to which \cite[Proposition 3]{MR1338691} relegates much of its proof) states that \Cref{eq:big.mdpt.map} splits over the entire punctured ball $E^{\times}_{\le 1}$ provided the sphere $\partial E_{\le 1}$ admits continuous self-maps with neither fixed nor negated points (i.e. $\varphi(x)\ne \pm x$). 
  \end{enumerate}
  Observe that a choice of a bisecting chord $[p_1,p_2]$, $p_i\in \partial E_{\le 1}$ for $p\in E^{\times}_{\le 1}$ determines both
  \begin{itemize}[wide]
  \item a 2-plane $\pi$ passing through $p$, namely that spanned by the origin and the $p_i$;

  \item as well as an orientation for that plane (the one, say, whereby the positive traversal of the bounding topological circle $\pi\cap \partial E_{\le 1}$ meets $p_1$, the ray through $p$ and $p_2$ in this order). 
  \end{itemize}
  Conversely, by the strict convexity which ensures (as noted in the proof of \Cref{pr:extr.pt.spn}) the uniqueness of a $(p\ne 0)$-bisected chord in a 0-centered 2-dimensional convex body, a splitting of \Cref{eq:big.mdpt.map} over $Y\subseteq E^{\times}_{\le 1}$ means \emph{precisely} a continuous section over $Y$ of the \emph{fiber bundle} (in the sense \cite[post Proposition IV.3.4]{lang-fund} of possibly infinite-dimensional Banach manifolds)
  \begin{equation}\label{eq:or.2pl.bdl}
    P
    \xrightarrowdbl{\quad\theta\quad}
    E^{\times}_{\le 1}
    ,\quad
    \theta^{-1}(p)
    :=
    \left\{\text{oriented $2$-planes through $p$}\right\}.
  \end{equation}
  That such a section exists in case \Cref{item:re:cont.bis:ball.min.segm} above, for instance, follows immediately from the contractibility of $Y:=E_{\le 1}\setminus [0,q]$ and the fact that locally trivial fiber bundles over contractible paracompact spaces are trivial \cite[Theorem 14.4.1]{td_alg-top}.

  As for \Cref{item:re:cont.bis:shll}, simply observe that a self-map $\varphi$ of $\partial E_{\le 1}$ with no fixed or negated points provides a (global) section to \Cref{eq:or.2pl.bdl}:
  \begin{equation}\label{eq:glob.sect}
    E^{\times}_{\le 1}
    \ni
    p
    \xmapsto{\quad}
    \spn\left\{p,\varphi\left(\frac{1}{\|p\|}p\right)\right\},
  \end{equation}
  with the orientation making the two vectors listed in \Cref{eq:glob.sect} into a positive basis, in that order.
\end{remark}

\pf{th:dimn.iff}
\begin{th:dimn.iff}
  The diagram
  \begin{equation*}
    \begin{tikzpicture}[>=stealth,auto,baseline=(current  bounding  box.center)]
      \path[anchor=base] 
      (0,0) node (l) {\Cref{item:th:dimn.iff:all.e.2}}
      +(2,.5) node (u) {\Cref{item:th:dimn.iff:all.e.n}}
      +(2,-.5) node (d) {\Cref{item:th:dimn.iff:1.e.2}}
      +(4,0) node (r) {\Cref{item:th:dimn.iff:1.e.n}}
      +(-4.5,0) node (ll) {\Cref{item:th:dimn.iff:dim}}
      +(-4,2) node (lul) {strong \Cref{item:th:dimn.iff:all.shl}}
      +(-1,2.5) node (luu) {strong \Cref{item:th:dimn.iff:1.shl}}
      +(-2.6,1.1) node (ld) {weak \Cref{item:th:dimn.iff:all.shl}}
      +(-.5,1.5) node (lr) {weak \Cref{item:th:dimn.iff:1.shl}}
      ;
      \draw[-implies,double equal sign distance] (l) to[bend left=6] node[pos=.5,auto] {$\scriptstyle $} (u);
      \draw[-implies,double equal sign distance] (u) to[bend left=6] node[pos=.5,auto] {$\scriptstyle $} (r);
      \draw[-implies,double equal sign distance] (l) to[bend right=6] node[pos=.5,auto,swap] {$\scriptstyle $} (d);
      \draw[-implies,double equal sign distance] (d) to[bend right=6] node[pos=.5,auto,swap] {$\scriptstyle $} (r);
      \draw[-implies,double equal sign distance] (ll) to[bend right=0] node[pos=.5,auto,swap] {$\scriptstyle \text{\cite[proof of Theorem II]{MR229004}}$} (l);
      \draw[-implies,double equal sign distance] (lul) to[bend left=6] node[pos=.5,auto,swap] {$\scriptstyle $} (luu);
      \draw[-implies,double equal sign distance] (lr) to[bend left=16] node[pos=.5,auto,swap] {$\scriptstyle $} (r);
      \draw[-implies,double equal sign distance] (lul) to[bend right=6] node[pos=.5,auto,swap] {$\scriptstyle $} (ld);
      \draw[-implies,double equal sign distance] (ld) to[bend right=6] node[pos=.5,auto,swap] {$\scriptstyle $} (lr);
      \draw[-implies,double equal sign distance] (luu) to[bend left=6] node[pos=.5,auto,swap] {$\scriptstyle $} (lr);
    \end{tikzpicture}
  \end{equation*}
  summarizes the more easily-disposed-of implications, with the unmarked arrows obvious and the labeled one relying on the fact that the proof of \cite[Theorem II]{MR229004} referenced above does not hinge crucially on $E$ being a Hilbert space (that source's standing assumption). Two other items will round out the argument.
  
  \begin{enumerate}[label={},wide]
  \item\textbf{\Cref{item:th:dimn.iff:1.e.n} $\Rightarrow$ \Cref{item:th:dimn.iff:dim}:} Reprise the proof of \Cref{th:2d.3pts}\Cref{item:th:2d.3pts:dim.constr} with the distinction that the $f_i$ now take values in the punctured ball $E^{\times}_{\le 1}$ instead of its bounding sphere. The prior argument applies to the functions
    \begin{equation*}
      X\ni
      x
      \xmapsto{\quad}
      \frac 1{\|f_i(x)\|}
      f_i(x)
      \in
      \partial E_{\le 1}
    \end{equation*}
    respectively retracting $f_i$ into the sphere. 

  \item\textbf{\Cref{item:th:dimn.iff:dim} $\Rightarrow$ strong \Cref{item:th:dimn.iff:all.shl}:} Fix a continuous $X\xrightarrow{f}E_{\le 1}$, $0<r<1$ and
    \begin{equation*}
      \text{0-symmetric strictly-convex body }
      K
      \subset
      E_{<1}
      ,\quad
      \left(E_{< 1}\setminus K\right)
      \subset E_{(r,1)},
    \end{equation*}
    and set $X_K:=f^{-1}K$. Being closed in a $(\le n-1)$-dimensional normal space and hence \cite[Corollary 3.6.3]{prs_dim_1975} again $(\le n-1)$-dimensional, $X_K$ affords a convex-combination expression
    \begin{equation}\label{eq:fxk.cvx}
      f|_{X_K}
      =
      \sum_i \lambda_i f_i
      ,\quad
      X_K\xrightarrow{\quad f_i\quad}
      \partial K
    \end{equation}
    by \cite[Corollary 7]{MR1338691} (applicable, given the \emph{strict} convexity of $K$). The selfsame strict convexity also forces
    \begin{equation*}
      \forall i
      \ :\
      f_i=f
      \text{ on }
      \partial K\subset X_K,
    \end{equation*}
    so we can simply extend all $f_i$ continuously across $X$ by $f_i:=f$ on $E_{\le 1}\setminus K$. Plainly, \Cref{eq:fxk.cvx} holds throughout $X$.  \qedhere
  \end{enumerate}
\end{th:dimn.iff}

\addcontentsline{toc}{section}{References}
%\bibliography{bib}{}
%\bibliographystyle{plain}

\def\polhk#1{\setbox0=\hbox{#1}{\ooalign{\hidewidth
  \lower1.5ex\hbox{`}\hidewidth\crcr\unhbox0}}}
  \def\polhk#1{\setbox0=\hbox{#1}{\ooalign{\hidewidth
  \lower1.5ex\hbox{`}\hidewidth\crcr\unhbox0}}}
  \def\polhk#1{\setbox0=\hbox{#1}{\ooalign{\hidewidth
  \lower1.5ex\hbox{`}\hidewidth\crcr\unhbox0}}}
  \def\polhk#1{\setbox0=\hbox{#1}{\ooalign{\hidewidth
  \lower1.5ex\hbox{`}\hidewidth\crcr\unhbox0}}}
  \def\polhk#1{\setbox0=\hbox{#1}{\ooalign{\hidewidth
  \lower1.5ex\hbox{`}\hidewidth\crcr\unhbox0}}}
  \def\polhk#1{\setbox0=\hbox{#1}{\ooalign{\hidewidth
  \lower1.5ex\hbox{`}\hidewidth\crcr\unhbox0}}}

\Addresses

\end{document}